\documentclass[12pt, reqno]{amsart}

\usepackage{amssymb}
\usepackage{amsmath}
\usepackage{amscd}
\usepackage{float}
\usepackage{graphicx}
\usepackage{amsfonts}
\usepackage{pb-diagram}
\usepackage[utf8]{inputenc}
\usepackage[T1]{fontenc}

\numberwithin{equation}{section}
\newtheorem{theorem}{Theorem}[section]
\newtheorem{corollary}[theorem]{Corollary}
\newtheorem{lemma}[theorem]{Lemma}
\newtheorem{proposition}[theorem]{Proposition}
\newtheorem{definition}[theorem]{Definition}
\newtheorem{remark}[theorem]{Remark}
\newtheorem{example}[theorem]{Example}

\newtheorem{conjecture}[theorem]{Conjecture}
\begin{document}

\title{A Partition Temperley--Lieb Algebra \newline \tiny{\sc (Work in progress)}}

\author{Jes\'us Juyumaya}
\address{Departamento De Matem\'{a}ticas, Universidad de Valpara\'{i}so, 
Gran Breta\~{n}a 1091, Valpara\'{i}so, Chile.}
\email{juyumaya@uvach.cl, juyumaya@gmail.com}

\thanks{This research has been supported in part by DIUV Grant Nº1-2011. Also, 
 co-financed by the European Union (European Social
 Fund - ESF) and Greek national funds through the Operational Program
 "Education and Lifelong Learning" of the National Strategic Reference
 Framework (NSRF) - Research Funding Program: THALES: Reinforcement of the
 interdisciplinary and/or inter-institutional research and innovation. 
}



\date{April 1, 2013}

\begin{abstract}
We introduce a generalization of the Temperley--Lieb algebra. This generalization 
is defined by adding  certain relations to the  algebra of braids and ties. A specialization  of this last algebra corresponds to one  small Ramified Partition algebra, this fact is the motivation for the name of our generalization.
\end{abstract}
\maketitle

\section*{Introduction}

The Temperley--Lieb algebra appears originally in  Statistical Mechanics  as well as in  Knot theory, quantum groups and  subfactors of von Neumann algebras. This algebra was discovered by Temperley and Lieb  by building  transfer matrices\cite{tl}.  Further, this algebra was rediscovered  by V. Jones\cite{jo83} who used it  in the construction of his polynomial invariant  for knots known as the Jones polynomial\cite{jo}.

\smallbreak

From a purely algebraic point of view, the Temperley--Lieb algebra is a quotient of the Iwahori--Hecke algebra by the  two--sided ideal generated by the Steinberg 
 elements $h_{ij}$ associated to $h_i$'s, where $\vert i-j \vert =1$ and $h_i$'s denote the usual generators of the Iwahori--Hecke algebra, view p. 35\cite{gohajo}. In other words, the Temperley--Lieb algebra can be defined by the usual presentation of the Iwahori--Hecke algebra but  by adding the relations 
 $h_{ij}=0$, for all $\vert i-j \vert =1$. Using this point of view, there are several generalizations of the Temperley--Lieb algebra, e.g.  see   \cite{fan, gojula}. This paper proposes a generalization of the Temperley--Lieb algebra  by adding relations  of  Steinberg types to the {\it algebra of braid and ties}\cite{aj, ry}.

\smallbreak

The algebra of braid and ties   ${\mathcal E}_n(u)$, where $u$ is a parameter and  $n$ denotes a positive integer, can be regarded as a generalization of the Hecke algebra and recently  E. O. Banjo proved that ${\mathcal E}_n(1)$ is isomorphic to a small ramified partition algebra, see Theorem 4.2\cite{ba}.  The possible connexion of the ${\mathcal E}_n(u)$ and the Partition algebras \cite{joPA, mar1} was speculated first by S. Ryom--Hansen\cite{ry}. The algebra ${\mathcal E}_n(u)$ is defined by two sets  of  generators and relations. One set of generators $T_1,\ldots , T_{n-1}$ reflects the braid generators of the Yokonuma--Hecke algebra\cite{yo, th, chda} of type $A$ and the other set of generators  $E_1, \ldots ,E_{n-1}$ reflects the behavior of the monoid $\mathsf{P}_n$ associated to the set  partitions of $\{1, \dots , n\}$. 
Thus, ${\mathcal E}_n(u)$ also can be thought as a $u$--deformation of an  amalgam among the symmetric group on $n$ symbols and $\mathsf{P}_n$.

\smallbreak

In short, in this paper we define and study the {\it Partition Temperley--Lieb algebra}, denoted ${\rm PTL}_n(u)$,  which is defined  by adding to the presentation  of ${\mathcal E}_n(u)$ mentioned above the following relations 
$$
E_iE_jT_{ij}=0 \quad \text{for all}\quad \vert i-j\vert=1
$$
where $T_{ij}$ is  the Steinberg element associated to the $T_i$'s.

\smallbreak

This work is organized as follows.  In  Section 1 we fix notations and  we recall the definition of the  Jimbo representation.  In Section 2 we recall the definition of the algebra ${\mathcal E}_n(u)$, we have included also some results from \cite{ry} which are used in the paper. In Section 3 we construct a non--faithful  tensor representation of the algebra ${\mathcal E}_n(u)$ which is used in Section 4 for the definition of our  Partition  Temperley--Lieb algebra ${\rm PTL}_n(u)$.   The Section  5 shows two presentations of the  ${\rm PTL}_n(u)$. By using one of these presentations we constructed a span linear  set of ${\rm PTL}_n(u)$ which is conjectured that is a basis for the Partition Temperley--Lieb algebra. Finally,  based on  a conjecture that  the algebra  ${\mathcal E}_n(u)$ supports a Markov trace, we prove in Section 7 under which  condition this trace could  pass to  ${\rm PTL}_n(u)$.

\section{Preliminaries}

Along the paper algebra means unital associative algebra, with unity $1$, over the field of rational function $K:={\Bbb C}(\sqrt{u})$ in the variable   $\sqrt{u}$. Consequently, we put  $u = (\sqrt{u})^2$.

\smallbreak

Let  $ {\rm H}_n = {\rm H}_n(u)$ be the Iwahori--Hecke algebra of type $A$, that is,  the algebra presented by generators $1, h_1, \ldots , h_{n-1}$ subject to braid relations among the $h_i$'s and the quadratic relation $h_i^2  =   u + (u-1)h_i$, for all $i$. 

We shall recall the Jimbo representation of the Hecke algebra. Set $V$ the $K$--vector space with basis $\{v_1, v_2\}$. Denotes by ${\bf J}$ the endomorphism of $V\otimes V$ defined through the mapping
$$
\begin{array}{ccl}
{\bf J}(v_i\otimes v_j) & = &  -v_i \otimes v_j \qquad \text{for } \quad i=j\\
{\bf J}(v_1\otimes v_2) & = & (u-1)\,v_1\otimes v_2 + \sqrt{u}\, v_2\otimes v_1  \\
 {\bf J} (v_2\otimes v_1) & = & \sqrt{u}\, v_1\otimes v_2.
\end{array}
$$
The Jimbo representation of ${\rm H}_n$ in $V^{\otimes n}$ is defined  by mapping $h_i\mapsto {\bf J}_i$, where ${\bf J}_i$  acts as the identity, with exception of the factors $i$ and $i+1$, where acts by ${\bf J}$.

\begin{proposition}\label{kerJ}
The kernel of the Jimbo representation is the two--sided ideal generated by $h_{ij}$, where  $\vert i-j\vert =1$ and
$$
h_{ij}:= 1 + h_i + h_j + h_ih_j  + h_jh_i + h_ih_jh_i.
$$
\end{proposition}

It is well known that the Temperley--Lieb algebra can be defined as  the quotient 
of the Iwahori--Hecke algebra by the Kernel of  Jimbo representation. Thus, the Temperley--Lieb algebra  can be defined by adding  extra non--redundant relations to the above presentations of the Hecke algebra. More precisely, we have the following definition.
\begin{definition}
The Temperley--Lieb algebra ${\rm TL}_n = {\rm TL}_n(u)$ is the algebra generated by $1, h_1, \ldots , h_{n-1}$ subject to the following relations:
\begin{eqnarray}
h_i^2 & =  & u + (u-1)h_i \qquad \text{ for all $i$}\label{tl1}\\
h_ih_j & =  & h_j h_i \qquad \text{ for $\vert i - j\vert >1$}\label{tl2}\\
h_ih_jh_i & =  & h_jh_i h_j \qquad \text{ for $\vert i - j\vert =1$}\label{tl3}\\
h_{ij} &  = & 0\qquad \text{ for $\vert i - j\vert =1$}.\label{tl4}
\end{eqnarray}
\end{definition}

It is well known that the  dimension of ${\rm TL}_n$ is the $n$th Catalan number $C_n: = \frac{1}{n+ 1}$\, $2n\choose{n}$ \cite{jo83} and that 
${\rm TL}_n$ has a presentation (reduced) with idempotents generators. Indeed, making 
$$
f_i := \frac{1}{1+u}(1+h_i)
$$
 we have the following proposition.
\begin{proposition}\label{pretl}
 ${\rm TL}_n$ can be presented by generators $1, f_1, \ldots ,f_{n-1}$ satisfying the following relations
 \begin{eqnarray}
f_i^2 & = & f_i \qquad \text{ for all $i$}\label{pretl1}\\
f_if_j & =  & f_j f_i  \qquad \text{ for $\vert i - j\vert >1$}\label{pretl2}\\
f_if_jf_i & =  & \frac{u}{(1+u)^ 2}  f_i\qquad \text{ for $\vert i - j\vert =1$}. \label{pretl3}
\end{eqnarray}
\end{proposition}

\smallbreak

By virtue  Proposition \ref{kerJ}, the Jimbo representation of the Iwahori--Hecke algebra defines a representation of the Temperley--Lieb algebra. In terms of the 
generators $f_i$'s, this representation, called also the Jimbo representation, acts on  $V^{\otimes n}$   by mapping $f_i\mapsto {\bf F}_i$. The  endomorphism ${\bf F}_i$  acts as the identity, with exception of the factors $i$ and $i+1$, where acts by ${\bf F}\in {\rm End}(V^{\otimes 2})$,
$$
\begin{array}{ccl}
{\bf F} (v_i\otimes v_j) & = &  0 \qquad \text{for } \quad i=j\\
{\bf F} (v_1\otimes v_2) & = & (u + 1)^{-1}(u\, v_1\otimes v_2 + \sqrt{u}\, v_2\otimes v_1)  \\
 {\bf F} (v_2\otimes v_1) & = & (u + 1)^{-1} (\sqrt{u}\,v_1\otimes v_2  + v_2\otimes v_1 ).
\end{array}
$$

\section{The algebra of braids and ties}

Let $\bf n$ be the poset $\{1, \dots , n\}$. A partition of $\bf n$ is a collection of pairwise disjoint subposets (called parts) whose union is equal to $\bf n$.  We shall denote $\mathsf{P}_n$ the set 
 formed by the partitions of $\bf n$. The cardinal $b_n$ of $\mathsf{P}_n$ is known as the $n$th Bell number.
 
\smallbreak 

Let $I\in\mathsf{P}_n$, an arc 
$i\frown  j $ 
of $I$ is an ordered pair $(i,j)\in \{1, \dots , n\}\times \{1, \dots , n\}$ such that 
\begin{enumerate}
\item $i<j$
\item $i$ and $j$ are in the same part of $I$
\item if $k$ is in the same part as $i$ and $i<k\leq j$, then $k=j$
\end{enumerate}

In other words the arcs are pairs of adjacent elements in each part of $I$. Therefore 
the elements of  $ \mathsf{P}_n$ can be represented by a  graph with  vertices  $\{1, \dots , n\}$ and whose edge connecting the vertices  $i$ and $j$ if and only if  $i\frown  j $ is an arc of $I$.  For example, for $n=3$ we have 
$$
\{\{1,2\}, \{3\}\} \qquad\text{is represented by} \qquad \begin{picture}(50,25)
\put(-2,2){$\bullet$}
\put(23,2){$\bullet$}
\put(-1, -5){\tiny{1}}
\put(24, -5){\tiny{2}}
\qbezier(0,5)(12,25)(25,5)
\put(-2,2){$\bullet$}
\put(48,2){$\bullet$}
\put(49, -5){\tiny{3}}
\end{picture}
$$
and so on. 

\smallbreak

The set  $ \mathsf{P}_n$ can  be regarded naturally as a poset, where the partial order  $\preceq$, is defined by:
$ I= (I_1, \dots I_k) \preceq J=(J_1, \dots J_l)$ if and only if each $J_i$ is a union of certain $I_i$'s. By using  $\preceq$ we give  to $ \mathsf{P}_n$ a  structure of commutative monoid by defining the product $I\ast J$, of $I$ with $J$, as the minimum  element of the poset  $ \mathsf{P}_n$  containing $I$ and $J$.
Clearly the unity is  $\{\{1\}, \{2\}, \dots , \{n\}\}\}$
which is represented by \quad $\begin{picture}(80,9)
\put(-2,2){$\bullet$}
\put(23,2){$\bullet$}
\put(73, 2){$\bullet$}
\put(-1, -5){\tiny{$1$}}
\put(24, -5){\tiny{$2$}}
\put(74, -5){\tiny{$n$}}
\put(40,3){\dots}
\end{picture}$. The monoid  $ \mathsf{P}_n$ 
 is generated by the unity and the elements:
$$
\begin{picture}(80,25)
\put(-2,2){$\bullet$}
\put(23,2){$\bullet$}
\put(48,2){$\bullet$}
\put(73, 2){$\bullet$}
\put(-1, -5){\tiny{$1$}}
\put(24, -5){\tiny{$i$}}
\put(44, -5){\tiny{$i+1$}}
\put(74, -5){\tiny{$n$}}
\qbezier(25,5)(37,25)(50,5)
\put(7,3){\dots}
\put(56,3){\dots}
\end{picture} \qquad \text{for all $1\leq i\leq n$}
$$ 

The  Hasse diagram for $ \mathsf{P}_3$  is:
$$
\begin{picture}(200,110)
\put(88,97){$\bullet$}
\put(113,97){$\bullet$}
\put(138,97){$\bullet$}
\put(89, 90){\tiny{1}}
\put(114, 90){\tiny{2}}
\put(139, 90){\tiny{3}}
\qbezier(115, 100)(127,120)(140,100)
\qbezier(90, 100)(102,120)(115,100)


\put(-2,47){$\bullet$}
\put(23,47){$\bullet$}
\put(48,47){$\bullet$}
\put(-1,40){\tiny{1}}
\put(24, 40){\tiny{2}}
\put(49, 40){\tiny{3}}
\qbezier(0, 50)(12,70)(25,50)


\put(88,47){$\bullet$}
\put(113,47){$\bullet$}
\put(138,47){$\bullet$}
\put(89, 40){\tiny{1}}
\put(114, 40){\tiny{2}}
\put(139, 40){\tiny{3}}
\qbezier(115, 50)(127,70)(140,50)


\put(178,47){$\bullet$}
\put(203,47){$\bullet$}
\put(228,47){$\bullet$}
\put(179, 40){\tiny{1}}
\put(204, 40){\tiny{2}}
\put(229, 40){\tiny{3}}
\qbezier(180, 50)(205,70)(230,50)


\put(88,-3){$\bullet$}
\put(113,-3){$\bullet$}
\put(138,-3){$\bullet$}
\put(89, -10){\tiny{1}}
\put(114, -10){\tiny{2}}
\put(139, -10){\tiny{3}}

\put(117,13){\line(0,1){18}}
\put(117,63){\line(0,1){18}}

\put(80,13){\line(-1,1){18}}
\put(150,13){\line(1,1){18}}

\put(60,63){\line(1,1){18}}
\put(170,63){\line(-1,1){18}}

\end{picture}
$$
\smallbreak

And we have, for  example:
$$
\begin{picture}(221,25)
\put(-2,2){$\bullet$}
\put(23,2){$\bullet$}
\put(48,2){$\bullet$}
\put(-1, -5){\tiny{1}}
\put(24, -5){\tiny{2}}
\put(49, -5){\tiny{3}}
\qbezier(0,5)(12,25)(25,5)
\put(63, 5){$\ast$}
\put(78,2){$\bullet$}
\put(104,2){$\bullet$}
\put(129,2){$\bullet$}
\put(79, -5){\tiny{1}}
\put(104, -5){\tiny{2}}
\put(129, -5){\tiny{3}}
\qbezier(105,5)(118,25)(133,5)
\put(155, 5){$=$}
\put(178,2){$\bullet$}
\put(204,2){$\bullet$}
\put(229,2){$\bullet$}
\put(179, -5){\tiny{1}}
\put(204, -5){\tiny{2}}
\put(229, -5){\tiny{3}}
\qbezier(180,5)(193,25)(208,5)
\qbezier(208,5)(221,25)(232,5)
\end{picture}
$$

As usual we denote $S_n$ the symmetric group on symbols and we  denote $s_i$ the transposition 
$(i,\, i+1)$. 

\smallbreak

For $I = \{ I_1, \dots , I_m\}\in \mathsf{P}_n$ and $w\in S_n$ we define $wI = \{wI_1,\ldots , wI_m\}$, where $wI_i$ is the subposet of $\bf n$ obtained by applying  $w$ to the elements of $I_i$.

\begin{definition}\rm 
We denote ${\mathcal E}_n={\mathcal E}_n(u)$ the algebra generated by $1, T_1, \ldots, T_{n-1}, E_1, \ldots , E_{n-1}$ satisfying the following relations:
\begin{eqnarray}
T_iT_j & = &  T_jT_i \qquad \text{for $\vert i-j \vert >1$}\label{E1}\\
T_iT_jT_i & = & T_jT_iT_i \qquad \text{ for $\vert i  -  j\vert =1$}\label{E2} \\
T_i^2 & = & 1 + (u-1) E_{i}  \left( 1+  T_i\right)\qquad \text{ for all $i$}\label{E3}\\ 
E_iE_j & = &   E_j E_i \qquad \text{ for all $i,j$}\label{E4}\\
E_i^2 & = & E_i \qquad \text{ for all $i$}\label{E5}  \\
E_iT_j & = &   T_j E_i \qquad \text{ for $\vert i  -  j\vert >1$}\label{E6}\\
E_iT_i & = &   T_i E_i \qquad \text{ for all $i$}\label{E7}\\
E_iE_jT_i  & = &  T_iE_iE_j \quad = \quad E_jT_iE_j\qquad \text{ for $\vert i  -  j\vert =1$}\label{E8}\\
E_iT_jT_i & = &   T_jT_i E_j \qquad \text{ for $\vert i - j\vert =1$}\label{E9}.
\end{eqnarray}

\end{definition}

If $w= s_{i_1}\cdots s_{i_k}\in S_n$ is reduced form for $w$, we write  $T_w := T_{i_1}\cdots T_{i_k}$ (this is a possible debt to a well known result of H. Matsumoto).

For $i<j$, we define $E_{ij}$ as
 $$
 E_{ij} = \left\{\begin{array}{ll}
 E_i & \text{for} \quad j = i +1\\ 
 T_i \cdots T_{j-2}E_{j-1}T_{j-2}^{-1}\cdots T_{i}^{-1}& \text{otherwise}
 \end{array}\right.
 $$

 For any $J=\{i_1, i_2, \ldots , i_k\}$ subposet of $\bf n$ we define $E_J=1$ if $k=1$ and 
$$
E_J := E_{i_1i_2}E_{i_2i_3}\cdots E_{i_{k-1}i_k} \quad \text{for}\quad k>1
$$
Note that $E_{\{i,j\}} = E_{ij}$. Also we note that 
in  Lemma 4\cite{ry} it is proved that $E_J$ can be computed as 
 $$
E_{J} = \prod_{j\in J,\, j\not= i_0}E_{i_0j}\qquad (i_0:= \text{min}\{i\,;\,i\in J\})
 $$
 
For $I = \{I_1, \ldots , I_m\}\in \mathsf{P}_n$,  we define  
 $E_I$ as
 $$
 E_I = \prod_{k}E_{I_k}
 $$

The Corollary 2\cite{ry} implies 
the following proposition.

\begin{proposition}
The mapping $E_i\mapsto 
\begin{picture}(80,25)
\put(-2,2){$\bullet$}
\put(23,2){$\bullet$}
\put(48,2){$\bullet$}
\put(73, 2){$\bullet$}
\put(-1, -5){\tiny{$1$}}
\put(24, -5){\tiny{$i$}}
\put(44, -5){\tiny{$i+1$}}
\put(74, -5){\tiny{$n$}}
\qbezier(25,5)(37,25)(50,5)
\put(7,3){\dots}
\put(56,3){\dots}
\end{picture}
$
defines a 
monoid  isomorphism between the monoid  generated by $1, E_1, \dots , E_{n-1}$ and  $\mathsf{P}_n$.
\end{proposition} 
 
\begin{proposition}[Corollary 1\cite{ry}]
For $I \in \mathsf{P}_n$  and $w\in S_n$, we have
$$
T_w E_I T_w^{-1} = E_{wI}.
$$
\end{proposition}

\begin{theorem}[Corollary 3\cite{ry}]\label{basEn}
The set ${\Bbb S}_n:=\{E_I T_w \,; \,w\in S_n,\, I\in  \mathsf{P}_n\}$ is a linear basis of 
${\mathcal E}_n$. Hence the dimension of ${\mathcal E}_n$ is $b_nn!$.

\end{theorem}
  
\section{A tensorial representation for ${\mathcal E}_n$}

In this section we will define a tensorial representation for ${\mathcal E}_n$. This representation  is nothing more than a  variation of  that constructed by S. Ryom--Hansen in Section 3\cite{ry}. We note that, contrary to the representation constructed by Ryom--Hansen, our variation is  a non--faithful  representation. This fact is the key point in order to define the Partition Temperley--Lieb algebra as a quotient of ${\mathcal E}_n$.
 
\smallbreak

 Let $V$ be  the $K$--vector space with basis $\{v_i^r \,;\, 1\leq i,r\leq n\}$, we define the endomorphisms ${\bf E}$ and $\bf{T}$  of $V^{\otimes 2}$
through the following mapping,
$$
{\bf E}(v_{i}^{r}\otimes v_{j}^{s}) := 
\left\{\begin{array}{lr}0 & \qquad \text{for } \quad r\not= s\\
v_{i}^{r}\otimes v_{j}^{s} &  \qquad \text{for } \quad r=s
\end{array}\right. 
$$
$$
{\bf T}(v_{i}^{r}\otimes v_{j}^{s}) := 
\left\{\begin{array}{ll}
-v_{j}^{s}\otimes v_{i}^{r} & \qquad \text{for } \quad r\not=s \\
-v_{i}^{r}\otimes v_{j}^{s}  & \qquad \text{for } \quad r=s, \,  i = j\\
 (u-1)\,v_i^r\otimes v_j^s + \sqrt{u}\, v_j^s\otimes v_i^r & \qquad \text{for } \quad r=s, \, i < j \\
  \sqrt{u}\, v_j^s\otimes v_i^r & \qquad \text{for } \quad r=s, \, i>j
\end{array}\right.
$$

Define now, ${\bf E}_i$ (respectively ${\bf T}_i$) as the endomorphism of $V^{\otimes n}$  that acts as the identity with exception on  the factors $i$ and $i+1$ where acts by ${\bf E}$ (respectively ${\bf T}$).

\begin{theorem}\label{JimboEn}
The mapping $E_i \mapsto {\bf E}_i$, $T_i \mapsto {\bf T}_i$ defines a  representation $\mathcal{J}_n$ of $\mathcal{E}_n$ in $V^{\otimes n}$.
\end{theorem}

\begin{proof}
The proof uses  the same strategy as Theorem 1\cite{ry}.
We  only need to check that the operators $ {\bf E}_i$ and  ${\bf T}_i$ satisfy the respective relations (\ref{E1})--(\ref{E9}). The relations (\ref{E1}), (\ref{E4})--(\ref{E7}) clearly hold. 

To check  relation (\ref{E3}) it is enough to take $n=2$. Evaluating the relation in $ v_i^r\otimes v_j^s$ with $r=s$, the relation becomes the Hecke quadratic relation.  In the case $r\not= s$, the operator ${\bf E}(1 + {\bf T})$ acts as zero  and  ${\bf T}^2$ as the identity, hence the relation holds.

To check the remaining  of the relations, without loss of  generality,  we can suppose $n=3$. 
 Also we observe that it is enough to check  the relations in question on the  basis elements $x= v_{i}^{r}\otimes v_{j}^{s}\otimes v_{k}^{t}$. By simplicity  we shall  introduce the following notation: whenever we have two repetitions in the 
upper indices in the basis elements, we omit the two repeated upper indices  and  we replace the remaining indices by a prime, e.g. $v_{i}^{r}\otimes v_{j}^{s}\otimes v_{k}^{r}$ is written simply as $v_{i}\otimes v_{j}^{\prime}\otimes v_{k}$. Then when we have two repetitions in the upper indices  we shall distinguish three forms of elements:
\begin{equation}\label{3form}
v_{i}^{\prime}\otimes v_{j}\otimes v_{k}
\qquad 
v_{i}\otimes v_{j}^{\prime}\otimes v_{k}
\qquad
v_{i}\otimes v_{j}\otimes v_{k}^{\prime}
\end{equation}
Further, in these elements we can suppose that the lower indices are $1$ or $2$ since ${\bf T}$ acts according the order in the pair formed by lower indices. 
Now, the action of ${\bf T}$ on  primed and unprimed elements is, up to  sign,  a transposition, so we can suppose that the lower index of the primed factor is always $1$. Therefore, the elements in the form as (\ref{3form}) can be reduced to consider the following cases: 
\begin{equation}\label{12form}
\begin{array}{lll}
v_{1}^{\prime}\otimes v_{1}\otimes v_{1}\qquad & v_{1}\otimes v_{1}^{\prime}\otimes v_{1}\qquad & v_{1}\otimes v_{1}\otimes v_{1}^{\prime}\\
v_{1}^{\prime}\otimes v_{1}\otimes v_{2}\qquad & v_{1}\otimes v_{1}^{\prime}\otimes v_{2} \qquad& v_{1}\otimes v_{2}\otimes v_{1}^{\prime}\\
v_{1}^{\prime}\otimes v_{2}\otimes v_{1} \qquad & v_{2}\otimes v_{1}^{\prime}\otimes v_{1} \qquad & v_{2}\otimes v_{1}\otimes v_{1}^{\prime}\\
v_{1}^{\prime}\otimes v_{2}\otimes v_{2} \qquad & v_{2}\otimes v_{1}^{\prime}\otimes v_{2}\qquad & v_{2}\otimes v_{2}\otimes v_{1}^{\prime}
\end{array}
\end{equation}

The checking of   (\ref{E8}) and (\ref{E9})  are similar and routine. Thus we shall check only the first relation of  (\ref{E8}). If all upper indices in $x$ are distinct the operator ${\bf E}_i{\bf E}_j$ acts as zero and as the identity if all upper indices are  equals. Hence
 ${\bf E}_1{\bf E}_2{\bf T}_1$ and  ${\bf T}_1{\bf E}_1{\bf E}_2$ coincide on such $x$'s. Now it is easy to check the relation whenever   $x$ is an element of (\ref{12form}) whose unprimed factor has equal lower indices. The checking  on the other elements of   (\ref{12form}) results from  a direct computation, e.g., for $x= v_{1}\otimes v_{2}\otimes v_{1}^{\prime}$ we have
$$
{\bf E}_1{\bf E}_2{\bf T}_1 (x)=(u-1){\bf E}_1{\bf E}_2(x) + \sqrt{u}\,{\bf E}_1{\bf E}_2( v_{2}\otimes v_{1}\otimes v_{1}^{\prime} )= 0 = {\bf T}_1{\bf E}_1{\bf E}_2(x)
$$
 
Finally we will check the relation (\ref{E2}). If in the basis elements the upper indices are all equal we are in the situation of Jimbo representation ${\bold J}$. If all upper indices are different the action becomes, up to sign, in the permutation action on the factors of the basis elements. Therefore, it only remains to check that   (\ref{E2}) is true when one evaluates on the elements of (\ref{12form}). Now, it is easy to see that the evaluation of both sides of (\ref{E2}) on the elements of (\ref{12form})  whose unprimed factors are equal is $-\sigma_{13}$, where $\sigma_{13}$ permutes the the first with the  third factor in the tensor product. The check of (\ref{E2})  on the remaining  elements of (\ref{12form}) is all similar for all. We shall do,  as a  representative case, the case  $x= v_{1}^{\prime}\otimes v_{1}\otimes v_{2}$:
\begin{eqnarray*}
{\bf T}_2{\bf T}_1{\bf T}_2(x)
 & = & 
  (u-1)\,{\bf T}_2{\bf T}_1(v_{1}^{\prime}\otimes v_{1}\otimes v_{2})  + \sqrt{u}\,{\bf T}_2{\bf T}_1(v_{1}^{\prime}\otimes v_{2}\otimes v_{1} )\\
& = & 
 -(u-1)\,{\bf T}_2(v_{1}\otimes v_{1}^{\prime}\otimes  v_{2})  - \sqrt{u}\,{\bf T}_2(v_{2}\otimes v_{1}^{\prime}\otimes  v_{1} )\\
 & = & 
 (u-1)\,(v_{1}\otimes  v_{2}\otimes  v_{1}^{\prime})  + \sqrt{u}\,( v_{2}\otimes v_{1}\otimes v_{1}^{\prime} )\\
 & = & 
 {\bf T}_1(v_{1}\otimes  v_{2}\otimes  v_{1}^{\prime}) \\
 & = & 
 -{\bf T}_1{\bf T}_2(v_{1}\otimes v_{1}^{\prime}\otimes v_{2}) = {\bf T}_1{\bf T}_2{\bf T}_1(x).
\end{eqnarray*}

\end{proof}

\section{The \rm{PTL} algebra}

We want to define a generalization  of Temperley--Lieb algebra by using the algebra $\mathcal{E}_n$, we shall call this generalization the Partition Temperley--Lieb algebra  which is denoted ${\rm PTL}_n$. A  first natural attempt  of definition ${\rm PTL}_n$ is as the algebra that results   by adding to defining relations of $\mathcal{E}_n$ the 
relations $T_{ij}= 0$, where $T_{ij}$ are  the Steinberg elements $T_{ij}$'s associated  to the $T_i$'s, 
$$
T_{ij} := 1 + T_i + T_j +T_iT_j + T_j T_i +T_iT_jT_j\quad \text{where }\qquad \vert i-j \vert =1
$$
As in  the classical case we want that the Jimbo representation  $\mathcal{J}$ of  $\mathcal{E}_n$ passes to  ${\rm PTL}_n$, hence  the $T_{ij}$'s must  be  killed by $\mathcal{J}$. But  unfortunately this does not happen. In fact,  for  $n=3$ and by taking $x= v_1\otimes v_2 \otimes v_1^{\prime}$, we have
$$
{\bf T}_1 x = (u-1)v_1\otimes v_2 \otimes v_1^{\prime} + \sqrt{u}\,v_2 \otimes v_1 \otimes v_1^{\prime}\qquad {\bf T}_2x = - v_1\otimes v_1^{\prime}\otimes v_2
$$
$$
{\bf T}_2{\bf T}_1 x = -(u-1)v_1\otimes v_1^{\prime} \otimes v_2 - \sqrt{u}\,v_2 \otimes v_1^{\prime} \otimes v_1\qquad {\bf T}_1{\bf T}_2x =  v_1^{\prime}\otimes v_1\otimes v_2
$$
$$
{\bf T}_1{\bf T}_2{\bf T}_1 x = (u-1)v_1^{\prime}\otimes v_1\otimes v_2  + \sqrt{u}\,v_1^{\prime} \otimes v_2 \otimes v_1
$$
Then 
\begin{eqnarray*}
(\mathcal{J} T_{12} )x & = & u\, v_1\otimes v_2 \otimes v_1^{\prime} - u\, v_1\otimes  v_1^{\prime} \otimes v_2 + \sqrt{u}\,v_2 \otimes v_1 \otimes v_1^{\prime} \\ 
& &  
- \sqrt{u}\,v_2 \otimes v_1^{\prime} \otimes v_1 + u\,v_1^{\prime}\otimes v_1\otimes v_2 + \sqrt{u}\,v_1^{\prime} \otimes v_2 \otimes v_1
\end{eqnarray*}
Therefore $\mathcal{J}$ does not kill $T_{12}$.

Having in mind the   above discussion we make the following definition.

\begin{definition}
The Partition Temperley--Lieb algebra ${\rm PTL}_n={\rm PTL}_n(u)$ is defined by adding to the  defining presentation of  ${\mathcal E}_n$ the relations:
\begin{equation}\label{rptl}
E_iE_jT_{i,j} = 0 \quad \text{for all}\quad \vert i -j \vert =1.
\end{equation}
\end{definition}
Clearly, from (\ref{E8}) we have that   $E_iE_jT_{i,j} = 0$ is equivalent to  $T_{i,j} E_iE_j= 0$.
\begin{remark}\rm
Notice that by taking $E_i=1$ the algebra ${\rm PTL}_n$  coincides with the classical Temperley--Lieb algebra. 
Also, we note that the defining relations  of ${\rm PTL}_n$ hold when  $T_i$ is replaced by  the generators $h_i$ of the  Temperley--Lieb algebra and $E_i$ is replaced  by 1, thus  the  mapping  $E_i\mapsto 1$ and $T_i\mapsto h_i$ defines an algebra homomorphism from ${\rm PTL}_n$  onto ${\rm TL}_n$.
\end{remark}

\begin{theorem}\label{JimboPTLn}
The Jimbo representation $\mathcal{J}_n$ of $\mathcal{E}_n$ factors through the algebra ${\rm PTL}_n$.
\end{theorem}
\begin{proof}
Without loss of  generality we can suppose that $n=3$. Thus, we must  prove 
that $\mathcal{J}_3 (E_1E_2T_{12})=0$. Now,  keeping the notations used during the proof 
of Theorem \ref{JimboEn}, to prove the theorem it is enough to see  that $\mathcal{J}_3 (E_1E_2T_{12})$ kill the basis elements  $x= v_{i}^{r}\otimes v_{j}^{s}\otimes v_{k}^{t}$. If all upper indices in $x$ are equal, $\mathcal{J}_3$ is the Jimbo representation of the Hecke algebra, so  $\mathcal{J}_3(T_{12})$  kill $x$;  hence  $\mathcal{J}_3 (E_1E_2T_{12})$ kill $x$ too. If the upper  indices of $x$ are not all equal, we have that $x$ is killed by   
$\bf{E}_1$ or $\bf{E}_2$, hence $\mathcal{J}_3 (E_1E_2T_{12})(x)=0$.
\end{proof}

We are going to prove now that the set of relations (\ref{rptl}) can be reduced to only one. To do  this we need to introduce the following element $\Gamma$, 
$$
 \Gamma := T_1T_2\cdots T_{n-1}
 $$
 \begin{lemma}\label{gam}
 For all $1\leq i, j\leq n-1$  we have:
 \begin{enumerate}
 \item $T_i=\Gamma^{i-1}T_1\Gamma^{-(i-1)}$
 \item $T_{i, i+1}=\Gamma^{i-1}T_{1,2}\Gamma^{-(i-1)}$
 \item $E_{i} = \Gamma^{i-1}E_1\Gamma^{-(i-1)}$
 \item $T_{i+1}\Gamma^{i-1} = \Gamma^{i-1}T_2 $
 \item $E_{\{i, i+2\}} = \Gamma^{i-1}E_{\{1,3\}}\Gamma^{-(i-1)}$
 \end{enumerate}
 \end{lemma}
\begin{proof}
The statement (1) results from an inductive argument on $i$ and the braid relations of $T_i$'s. The statement (2)  is a  result applying (1). The proof of statement (3) is analogous to the proof of (1), that is: an argument inductive on $i$ and using the relation (\ref{E6}). The  statement (4) is clear, since (1). Finally, we have: 
\begin{eqnarray*}
\Gamma^{i-1}E_{\{1,3\}}\Gamma^{-(i-1)} &  =&  \Gamma^{i-1}T_2 E_1 T_2^{-1}\Gamma^{-(i-1)}\\
&  =& 
 \Gamma^{i-1}T_2 (\Gamma^{-(i-1)}E_i\Gamma^{(i-1)}) T_2^{-1}\Gamma^{-(i-1)}\\
 &=& 
  T_{i+1} E_iT_{i+1}^{-1}
\end{eqnarray*}
Thus, the  statement (5) is proved.
 \end{proof}
\begin{proposition}
The relation $E_1E_2T_{1,2} = 0$ implies the relations $E_iE_jT_{i,j} = 0$, for all  $\vert i-j\vert =1$.
\end{proposition}
\begin{proof}
 We can suppose $j=i+1$, since   $T_{ij}=T_{ji}$ and $E_i$ and $E_j$ commute. 
From the statements  (1) and (3)Lemma \ref{gam}, we have:
\begin{eqnarray*}
E_i E_{i+1} T_{i, i+1} &  = &   (\Gamma^{i-1}E_1\Gamma^{-(i-1)})(\Gamma^{i}E_1\Gamma^{-i})
(\Gamma^{i-1}T_{1,2}\Gamma^{-(i-1)}) \\
 & = & 
  \Gamma^{i-1}E_1 \Gamma E_1\Gamma^{-1} T_{1,2}\Gamma^{-(i-1)}) = \Gamma^{i-1}E_1E_2 T_{1,2}\Gamma^{-(i-1)})
\end{eqnarray*}
 Hence the proof follows.
\end{proof}
\begin{corollary}\label{PTLquo}
The Partition Temperley--Lieb algebra ${\rm PTL}_n$ can be regarded as the quotient 
of ${\mathcal E}_n$ by the two--sided ideal generated by $E_1E_2T_{12}$.
\end{corollary}

\section{Others presentations for ${\rm PTL}_n$}

 In order to have more comfortable notations we shall introduce the following element $\delta$, 
$$
\delta := \frac{1-u}{1+u}\in K
$$ 
\subsection{}
Having in mind the  definition of   the idempotents generators  $f_i$ of the Temperley--Lieb algebra, it is natural to consider the following definition.
$$
F_i := \frac{1}{u+1}(1 + T_i)\qquad (1\leq i \leq n-1)
$$
It is obvious that  $F_i$ commute with  $E_i$ (and $T_i$) and that they form a set of generators for the 
algebra ${\rm PTL}_n$, but notice that the $F_i$'s are not   idempotents. In fact, from (\ref{E3}) we have 
$$
F_i^2 =  \frac{1}{(u+1)^2}(1 + 2T_i + 1 + (u-1)E_i + (u-1)E_iT_i)
$$
then
$$
F_i^2 = (1+\delta)F_i - \delta E_iF_i
$$
We have the following proposition
\begin{theorem}\label{preptlF}
${\rm PTL}_n$ can be presented by the generators $1, E_1, \ldots , E_{n-1}, F_1, \ldots ,  F_{n-1}$ subject to the relations (\ref{E4}), (\ref{E5}) together with the following relations
\begin{eqnarray}
F_i^2 & = &  (1+\delta)F_i - \delta E_iF_i \label{ptlF4}\\
F_i F_j  & = &  F_j F_i \quad \text{ for all $\vert i-j\vert >1$}  \label{ptlF1}\\
F_i E_j  & = &  E_j F_i \quad \text{ for all $\vert i-j\vert >1$}  \label{ptlF2}\\
E_i F_i   & = &  F_i E_i  \label{ptlF3}
\end{eqnarray}
and for all $\vert i  -  j\vert =1$:
\begin{eqnarray}
E_iE_jF_i  & = &   \quad  F_i E_iE_j	\quad  = \quad  E_jF_iE_j + \frac{1}{u+1} (E_iE_j - E_j)\label{ptlF5}\\
E_iF_jF_i 
&  =  &   
F_jF_iE_j + \frac{1}{u+1}\left[(E_i  - E_j)F_j + F_i(E_i-E_j)\right]- \frac{1}{(u+1)^2}(E_i-E_j) \label{ptlF6}\\
 F_iF_jF_i
& = & 
\frac{1}{(u+1)^2}\left(F_i - (1-u)E_iF_i\right)\label{ptlF7}
\end{eqnarray}
\end{theorem}

\begin{proof}
It is easy to check that (\ref{E1}) (respectively (\ref{E6})) is equivalent to (\ref{ptlF1}) (respectively (\ref{ptlF2})), so    having in mind the   previous discussion to the theorem, it only remains to prove that the relations (\ref{ptlF5})--(\ref{ptlF7}) hold and that relations  (\ref{E8}), (\ref{E9}), (\ref{rptl}) and (\ref{E2}) can be deduced  from the relations 
of the theorem.

We have that $T_i= (u+1)F_i -1$. Now replacing  this expression of $T_i$ in  (\ref{E8}) (respectively (\ref{E9})) it is a routine to check that (\ref{E8}) becomes (\ref{ptlF5}) (respectively (\ref{ptlF6})). 

We have to  check that relation (\ref{rptl}) is equivalent to relation (\ref{ptlF7}). We have
$$
T_iT_j = ((u+1)F_i -1)((u+1)F_j -1) = (u+1)^2 F_iF_j - (u+1)F_i - (u+1)F_j +1
$$
then 
\begin{eqnarray*}
 T_iT_jT_i 
 & = & 
  (u+1)^3 F_iF_jF_i - (u+1)^2F_i^2 - (u+1)^2F_jF_i + (u+1)F_i\\
&  & 
 -(u+1)^2 F_iF_j + (u+1)F_i + (u+1)F_j -1
\end{eqnarray*}
Therefore, by using (\ref{ptlF4}), we deduce
\begin{eqnarray*}
 T_iT_jT_i & = &  (u+1)^3 F_iF_jF_i + (1-u^2)E_iF_i \\
 &  &
  - (u+1)^2F_jF_i - (u+1)^2F_iF_j + (u+1)F_j- 1
 \end{eqnarray*}
Now, substituting each summand of $T_{ij}$ by its expression in term of $F_i$'s one obtains    
$$
T_{ij} = (u+1)^3 F_iF_jF_i + (1-u^2)E_iF_i - (u+1)F_i
$$
Hence (\ref{rptl}) is equivalent  (\ref{ptlF7}).

Finally notice that (\ref{ptlF7}) implies (\ref{E2}), since  the above expression 
of $T_iT_jT_i$ in terms of $F_i$'s tells us that  (\ref{E2}) is equivalent to 
$$
(u+1)^2 F_iF_jF_i + (1-u)E_iF_i + F_j  = 
(u+1)^2 F_jF_iF_j + (1-u)E_jF_j + F_i
$$
Thus the proof is concluded.
\end{proof}

\subsection{}

In this subsection we shall show   a presentation of ${\rm PTL}_n$ by idempotent generators.  For  $1\leq i <j\leq n-1$, we define
$$
L_i := \frac{1}{1+u}\left( T_i + 1\right) \left(\alpha +(1-\alpha)E_i\right)\qquad \text{where}\quad \alpha := \frac{1+u}{2}
$$
notice that 
\begin{equation}\label{adL}
L_i =\frac{1}{2}\left(T_i + \delta T_iE_i + \delta E_i +  1 \right) =
\frac{1}{2}(1+T_i)(1+ \delta E_i)
\end{equation}
Also we have
\begin{equation}\label{L->F}
L_i = \frac{u+1}{2}F_i  + \frac{1-u}{2}E_iF_i
\end{equation}

It is clear that $L_i$ commute with $E_i$,  $T_i$ and $F_i$. We have the following useful lemma.
\begin{lemma}\label{ide}
 For all $i$ we have:
\begin{enumerate}
\item $L_i^2 = L_i$
\item $(1+u)E_iL_i = E_i(1+ T_i)$
\item $T_i = 2L_i +(u-1)E_iL_i -1$
\item $E_iL_i = E_iF_i $
\item $F_i = (1+\delta)L_i -\delta E_iL_i $.
\end{enumerate}
\end{lemma}
\begin{proof} 
We have:
$$
L_i^2 = 4^{-1} (1+T_i)^2 (1+\delta E_i)^2 = 4^{-1}(2(1+ T_i) +(u-1)E_i(1+T_i))(1 +(2\delta+ \delta^2)E_i)
$$
then
\begin{eqnarray*}
L_i^2  & = & 4^{-1}(1+ T_i)(2 +(u-1)E_i)(1 +(2\delta+ \delta^2)E_i)\\
        & = & 4^{-1}(1+ T_i)(2+ (2(2\delta+ \delta^2) + (u-1) + (u-1)(2\delta +                    \delta^2))E_i)\\
       & = & 4^{-1}(1+ T_i)(2 + (2\delta + \delta^2)(1+u) + u-1 )E_i) \\
        & = & 4^{-1}(1+ T_i)(2 + 2\delta E_i) = L_i.
\end{eqnarray*}
The second assertion  follows by  multiplying  the  formula of $L_i$ by $E_i$. To prove the third assertion,  we bring first $E_iT_i$ from the second assertion and then  we substitute this expression of $E_iT_i$ in (\ref{adL}), thus the third assertion follows. The fourth assertion results by multiplying (\ref{L->F}) by $E_i$. The fifth assertion result directly from (4) and (\ref{L->F}).
\end{proof}

\begin{theorem}\label{preptl}
${\rm PTL}_n$ can be presented by the generators $1, E_1, \ldots , E_{n-1}, L_1, \ldots ,  L_{n-1}$ subject to the relations (\ref{E4}), (\ref{E5}) together with the following relations
\begin{eqnarray}
L_i^2   & = &   L_i \label{ptl1}\\
L_i L_j  & = &  L_j L_i \quad \text{ for all $\vert i-j\vert >1$}  \label{ptl2}\\
L_iE_j   & = &  E_j L_i  \quad \text{ for all $\vert i-j\vert >1$} \label{ptl3}\\
 L_i E_i  & = &  E_iL_i  \label{ptl4}
\end{eqnarray}
and for all $\vert i  -  j\vert =1$:
\begin{eqnarray}
& & 
E_iE_jL_i  \,  =   \,  L_i E_iE_j \, = \, E_jL_iE_j +2^{-1}(E_iE_j -E_j)\label{ptl5}\\
& & 4L_iL_jE_i +  2 E_j(L_j + L_i) + E_i \,  =  \,   4E_jL_iL_j + 2(L_i + L_j)E_i  + E_j \label{ptl6}
\end{eqnarray}
\begin{equation}\label{ptl7}
\begin{array}{c}
 8L_iL_jL_i + 4(u-1) \left[L_iE_jL_jL_i +E_iL_iL_jL_i+ L_iL_jE_iL_i\right]   \\
 + (u-1)^2(u+5)E_iE_j L_iL_jL_i \, = \, 2L_i + 3(u-1)E_iL_i + (u-1)^2E_iE_jL_i
\end{array}
\end{equation}
\end{theorem}

\begin{proof}
We will use the presentation of Theorem \ref{preptlF}. From (5)Lemma \ref{ide},  follows that   ${\rm PTL}_n$ is  generated by   $1$, $E_i$'s and  $L_i$'s. 
Checking that (\ref{ptlF4})--(\ref{ptlF6}) are equivalent, respectively, to (\ref{ptl1})--(\ref{ptl6}) is a straight forward and just a routine, so we leave the computation to the reader. Thus, 
to finish the proof  it only remains to check that (\ref{ptl7})is equivalent to (\ref{ptlF7}).

We have 
\begin{eqnarray*}
F_iF_j & = & ((1+\delta )L_i -\delta E_iL_i)((1+\delta )L_j -\delta E_jL_j) \\
& =&  (1+ \delta)^2L_iL_j -\delta (1+\delta)L_iE_jL_j - \delta (1+\delta)E_iL_iL_j +\delta^2 E_iL_iE_jL_j)
\end{eqnarray*}
Hence
\begin{eqnarray*}
F_iF_jF_i 
 & = &
(1+\delta)^3 L_iL_jL_i -\delta (1+\delta)^2 \left[ L_iE_jL_jL_i + E_iL_iL_jL_i + L_iL_jE_iL_i\right] \\
& &
  + \delta^2(1+\delta)E_iL_iE_jL_jL_i  + \delta^2(1+\delta)L_iE_jL_jE_iL_i + \delta^2(1+\delta)E_iL_iL_jE_iL_i \\ 
  & & 
  - \delta^3 E_iL_iE_jL_jE_iL_i
\end{eqnarray*}
Using now  (\ref{E4}), (\ref{E5}), (\ref{ptl4}) and (\ref{ptl5}) we get
\begin{eqnarray*}
F_iF_jF_i 
 & = &
(1+\delta)^3 L_iL_jL_i -\delta (1+\delta)^2 \left[ L_iE_jL_jL_i + E_iL_iL_jL_i + L_iL_jE_iL_i\right] \\
& &
  (2\delta^2(1+\delta)- \delta^3)E_iE_jL_iL_jL_i + \delta^2(1+\delta)E_iL_iL_jE_iL_i
\end{eqnarray*}
Now applying on the last term the relation (\ref{ptl4}) and  using later (\ref{ptl5}), we get  $E_iL_iL_jE_iL_i = L_i(E_iL_jE_i)L_i$, so
\begin{eqnarray*}
E_iL_iL_jE_iL_i & = & L_i \left[ E_iE_jL_j -\frac{1}{2}(E_iE_j - E_i) \right]L_i \\
& = & 
E_iE_jL_iL_jL_i - \frac{1}{2}E_iE_jL_i^2 +\frac{1}{2}L_iE_iL_i \quad (\text{by using (\ref{ptl5})})\\
& = & 
E_iE_jL_iL_jL_i - \frac{1}{2}E_iE_jL_i +\frac{1}{2}E_iL_i \quad (\text{by using (\ref{ptl1}) and (\ref{ptl4})})
\end{eqnarray*}
Then 
\begin{eqnarray*}
F_iF_jF_i 
&  = & 
(1+\delta)^3 L_iL_jL_i -\delta (1+\delta)^2 \left[L_iE_jL_jL_i + E_iL_iL_jL_i+L_iL_jE_iL_i\right] \\
& & 
  + (2\delta^3 + 3\delta^ 2)E_iE_jL_iL_jL_i 
  - \delta^2(1+\delta)\left[ \frac{1}{2}E_iE_jL_i -\frac{1}{2}E_iL_i\right]
\end{eqnarray*}
On the other side, from (4)Lemma \ref{ide}, we have 
$$
F_i+ (u-1)E_iF_i = (1+\delta )L_i - \delta E_iL_i + (u-1)E_iL_i = (1+\delta )L_i - (u+2) \delta E_iL_i
$$
Therefore, the relation (\ref{ptl7}) is equivalent to 
\begin{eqnarray*}
&   & 
(1+\delta)^3 L_iL_jL_i -\delta (1+\delta)^2 \left[L_iE_jL_jL_i +E_iL_iL_jL_i+ L_iL_jE_iL_i\right] 
  + (2\delta^3 + 3\delta^ 2)E_iE_jL_iL_jL_i \\  
& = & 
\frac{1}{(u+1)^2}\left[(1+\delta )L_i - (u+2) \delta E_iL_i \right] + \delta^2(1+\delta)\left[ \frac{1}{2}E_iE_jL_i -\frac{1}{2}E_iL_i\right]
\end{eqnarray*}
which is reduced, after multiplication by $(u+1)^2$, to
\begin{eqnarray*}
&   & 
\frac{8}{(u+1)} L_iL_jL_i -4\delta \left[L_iE_jL_jL_i +E_iL_iL_jL_i+ L_iL_jE_iL_i\right] 
  + (1-u)^2(2\delta +3)E_iE_jL_iL_jL_i \\  
& = & 
(1+\delta )L_i - (u+2) \delta E_iL_i + (1-u)^2(1+\delta)\left[ \frac{1}{2}E_iE_jL_i -\frac{1}{2}E_iL_i\right]
\end{eqnarray*}
or equivalently
\begin{eqnarray*}
&   & 
\frac{8}{(u+1)} L_iL_jL_i -4\delta \left[L_iE_jL_jL_i +E_iL_iL_jL_i+ L_iL_jE_iL_i\right] 
  + (1-u)^2(2\delta +3)E_iE_jL_iL_jL_i \\  
& = & 
(1+\delta )L_i - 3\delta E_iL_i  + (1-u)\delta E_iE_jL_i 
\end{eqnarray*}
Multiplying this last equation by $u+1$ we obtain (\ref{ptl7}).
\end{proof}

\smallbreak

\begin{remark} \rm
By taking $E_i=1$ the elements $L_i$'s become $f_i$'s and  the Theorem \ref{preptl} and Theorem \ref{preptlF} become Theorem \ref{pretl}.  
\end{remark}

\section{A linear basis for ${\rm PTL}_n$} 

By using  essentially Theorems \ref{preptlF}, \ref{JimboPTLn}  we shall  construct a linear basis of ${\rm PTL}_n$. Further we use also the following lemmas.
\begin{lemma}\label{FE=EF}
For all $i,j$ such that $\vert i-j \vert=1$, we have:
\begin{enumerate}
\item 
$F_iE_{j} = T_iE_jT_i^{-1} F_i + \frac{1}{u+1}\left(E_j - T_iE_jT_i^{-1}\right)$
\item 
$E_jF_i = F_i T_iE_jT_i^{-1} + \frac{1}{u+1}\left(E_j - T_iE_jT_i^{-1}\right)$
\end{enumerate}
\end{lemma}
\begin{proof}
It is enough  to expand $F_i$ in both side of the equality. 
\end{proof}
\begin{lemma}\label{aiju}
Any word in $1, F_1,\dots , F_{n-1}$, $E_1,\dots , E_{n-1}$ can be expressed as a $K$--linear combination of words in $E_i$'s and $F_i$'s having   at most one $F_{n-1}$, $E_{n-1}$, or $F_{n-1}E_{n-1}$.
\end{lemma}
\begin{proof}
It is a consequence of Proposition 1\cite{aj} and the fact that $F_i$ is a linear expression of $1$ and $T_i$.
\end{proof}
\begin{definition}
A word in $F_1, \dots , F_{n-1}$ is called $F$--reduced (or simply reduced) if and only if has the form 
\begin{equation}\label{Fred}
(F_{i_1}\cdots F_{j_1})(F_{i_2}\cdots F_{j_2})\cdots (F_{i_k}\cdots F_{j_k})
\end{equation}
 where $0\leq k\leq n-1$ and 
\begin{eqnarray*}
1\leq i_1 < i_2<\cdots <i_k\leq n-1\\
1\leq j_1 < j_2<\cdots <j_k\leq n-1\\
i_1\geq i_2, i_2\geq j_2 ,\ldots ,i_k\geq j_k
\end{eqnarray*} 
\end{definition}

\begin{proposition}
Any word in $1, E_1, \dots , E_{n-1}, F_1, \dots , F_{n-1}$ may be written as $K$--linear combination of words  in the form $E_IF$, where $I\in \mathsf{P}_n$ and $F$ is $F$--reduced.
\end{proposition}
\begin{proof}
We have  adapted  the proof of Proposition 2.8\cite{gohajo}. 
We will use induction on $n$. The assertion is clearly  valid for $n=2$. We assume now that the proposition is valid for $n$. Let $W$ a word in $1, E_1, \dots , E_n, F_1, \dots , F_n$. By using  Lemma \ref{FE=EF} we can move the $E_i$'s appearing in $W$ to the front position, obtaining in this way  that  $W$ is  a linear combination of words in the form $E_IW^{\prime}$, where $W^{\prime}$ is a word in $1, F_1, \ldots , F_n$. Thus, to prove the proposition it is  enough    to show that $W^{\prime}$ is a linear combination of words in the form desired. Now, if $W^{\prime}$ does not contain $F_n$ then we are done. If $W^{\prime }$ contains $F_n$, according to Lemma \ref{aiju}, we have that  
$W^{\prime}$ is a linear combination of words in the form
$$
W_1R_nW_2
$$
where $R_n = E_n$, $F_n$ or $E_nF_n$ and $W_i$ are words in $1, E_1, \dots , E_{n-1}, F_1, \dots , F_{n-1}$. If $R_n=E_n$, according to Lemma \ref{FE=EF}, we can move $R_n$ to the front position and then  using the induction hypothesis we are done.
 Suppose  $R_n = F_n$, we note that by induction hypothesis $W_2$ is a linear combination of words in the form 
 $$
 E_JV(F_nF_{n-1}\cdots F_{j_k})
 $$
 where now $J\in\mathsf{P}_{n-1}$,  $V$ is a word reduced in $1, F_1, \dots , F_{n-2}$ (notice that 
 $F_nF_{n-1}\cdots F_{j_k}$  could be empty). Hence $W^{\prime}$ is a linear combination of words of the form
 $$
W_1F_n  E_JV(F_nF_{n-1}\cdots F_{j_k})
 $$
Now,  $F_n  E_J= E_{s_nJ}F_n $, so using (\ref{ptlF4}) and  (\ref{ptlF2})  follows that $W^{\prime}$ can be written as  a linear combination
$
(1 + \delta )N_1 - \delta N_2
$ 
with  $N_1:= E_{J^{\prime}}V^{\prime}(F_nF_{n-1}\dots F_{j_k})$ and  $N_2:= E_{J^{\prime}}V^{\prime}(E_nF_nF_{n-1}\dots F_{j_k})$, 
where $J^{\prime}\in\mathsf{P}_n$ and  $V^{\prime}$ is a word  in $1, F_1, \dots F_{n-1}$. Again we note that in $N_2$,  $E_n$ can move to the front position, so $N_2$ is in fact in the form of $N_1$. Therefore, $W^{\prime}$ is a linear combination of words in the form
$$
 E_{J^{\prime}}V^{\prime}(F_nF_{n-1}\dots F_{j_k})
$$ 
where $J^{\prime}\in\mathsf{P}_n$ and  $V^{\prime}$ is a word  in $1, F_1, \dots F_{n-1}$. Applying the induction hypothesis, on $V^{\prime}$, we deduce that $W^{\prime}$ is a linear combination of words in the form 
$
E_I F
$, 
where $I\in\mathsf{P}_n$ and $F$ has the form 
$$
F= (F_{i_1}\cdots F_{j_1})(F_{i_2}\cdots F_{j_2})\cdots (F_{i_k}\cdots F_{j_k})
$$
with $i$'s increasing and $i_l\geq j_l$, for all $1\leq l\leq k $. Thus it  remains to prove that in $F$'s the $j$'s can be taken increasing.
 Suppose $j_1\geq j_2$, so 
$$
F  =  (F_{i_1}\cdots F_{j_1 + 1})(F_{i_2}\cdots (F_{j_1}F_{j_1 +1}F_{j_1})\cdots F_{j_2})\cdots (F_{i_k}\cdots F_{j_k})
$$
Then, by using (\ref{ptlF7}), we have $F = (u + 1)^{-2}F_1 -  (u+1)^{-1}\delta F_2$, where 
$$
F_1 := (F_{i_1}\cdots F_{j_1 + 1})(F_{i_2}\cdots F_{j_1}\cdots F_{j_2})\cdots (F_{i_k}\cdots F_{j_k}) 
$$
and 
$$
F_2:=(F_{i_1}\cdots F_{j_1 + 1})(F_{i_2}\cdots (E_{j_1}F_{j_1})\cdots F_{j_2})\cdots F_{i_k}\cdots F_{j_k})
$$
Clearly (applying Lemma \ref{FE=EF}),  $E_{j_1}$ in $F_2$ can be moved to the front position. Therefore, by using an inductive argument we deduce  that  $F$ can be expressed as a linear combination in the desired form. Hence $W^{\prime}$ can be written in the desired form. Thus, the proof 
is concluded. 
\end{proof}

\begin{conjecture}
The set formed by the elements $E_IF$, where $I\in \mathsf{P}_n$ and $F$ is reduced ,  is a linear basis of ${\rm PTL}_n$. Hence 
 the dimension of ${\rm PTL}_n$ is $b_nC_n$.
\end{conjecture}

\section{Markov trace}
For $d$  a positive integer we denote ${\rm Y}_{d,n}={\rm Y}_{d,n}(u)$ the Yokonuma--Hecke algebra, i.e. the algebra presented  by braid generators $g_1, \ldots , g_{n-1}$ together with the framing generators $t_1, \ldots , t_{n}$ which   satisfies the following defining  relations:
 braids relation (of  type $A$) among the $g_i$'s, $t_it_j=t_jt_i$, 
$g_it_j = t_{s_i(j)}g_i$  and 
$$
g_i^2 = 1 + (u-1) e_i (1  + g_i)
$$
where $e_i$ is defined  as
$$
e_i :=\frac{1}{d}\sum_{s=1}^{d}t_i^st_{i+1}^{-s}
$$
\begin{proposition}
We have a natural algebra morphism $\psi : 
{\mathcal E}_n\mapsto {\rm Y}_{d,n}$ defined through the mapping $T_i\mapsto g_i$ and 
$E_i  \mapsto e_i$.
\end{proposition}
\begin{proof}
According  to Lemma 2.1\cite{jula3} the defining relations of ${\mathcal E}_n$ are satisfied by changing $T_i$ by $g_i$ and $E_i$ by $e_i$. Hence the proof follows.
\end{proof}

\begin{theorem}[See \cite{ju}]\label{trace}
Let $z$, $x_1, \ldots , x_{d-1}$ be in ${\Bbb C}$.
There exists a unique family of linear map $ \{{\rm tr}_n\}_n $ on inductive limit associated  to the family $\left\{{\rm Y}_{d,n}\right\}_n$ with
values  in  ${\Bbb C}$ satisfying the rules:
$$
\begin{array}{rcll}
{\rm tr}_n(ab) & = & {\rm tr}_n(ba) & \\
{\rm tr}_n(1) &  = & 1 & \\
{\rm tr}_{n+1}(ag_n) & = & z\,{\rm tr}_n(a) & \qquad {\text for}\quad  a \in {\rm Y}_{d,n} \\
{\rm tr}_{n+1}(at_{n+1}^m) & = & x_m{\rm  tr}_n(a) & \qquad   {\text for} \quad a \in {\rm Y}_{d,n}, \, 1\leq m\leq d-1.
\end{array}
$$
\end{theorem}

It is natural to consider the composition ${\rm tr}_n\circ \psi$ which  defines a Markov trace on  ${\mathcal E}_n$. This  supports  the following conjecture.
\begin{conjecture}\label{conaj}[Aicardi, Juyumaya]
The algebra ${\mathcal E}_n$ supports a Markov trace. I.e. for all $n\in {\Bbb N}$ we have a unique  linear map $\rho_n :{\mathcal E}_n\longrightarrow K( A, B)$  such that  for all $x, y\in {\mathcal E}_n$, we have:
\begin{enumerate}
\item $\rho_n(1) = 1$
\item $\rho_n(xy) = \rho_n(yx)$
\item $\rho_{n+1}(xT_n) =\rho_{n+1}(xE_nT_n)=A \rho_n(x)$
\item $\rho_{n+1}(xE_n) =B \rho_n(x)$
\end{enumerate}
where $A$ and $B$ are parameters.
\end{conjecture}
\begin{example}
According to the rule (3)Conjecture\ref{conaj} of $\rho$ we have, $\rho(E_1T_1T_2T_1)= A\rho(E_1T_1^2)$. Now, 
$E_1T_1^2 = E_1 (1 + (u-1)E_1(1 + T_1))=  uE_1 + (u-1) E_1T_1$. So
$$
\rho(E_1T_1T_2T_1)  = A(uB  + (u-1)A) = u AB + (u-1)A^2.
$$
\end{example}

Assuming that  Conjecture \ref{conaj} is true, we are going to study when the Markov trace $\rho_n$ passes to ${\rm PTL}_n$. According to Corollary \ref{PTLquo}, studying    the  factorization  of  $\rho_n$ to ${\rm PTL}_n$ is reduced to  studying the values of $\rho_n$ on the two--sided ideal generated by $E_iE_jT_{12}$. For this study we need
the following lemmas. 
\begin{lemma}\label{TwT12}
\begin{enumerate}
\item $T_1T_{12} = \left[ 1 + (u-1)E_1\right]  T_{12}$
\item $T_2T_{12} = \left[ 1 + (u-1)E_2\right] T_{12}$
\item $T_1T_2T_{12} = \left[ 1  + (u-1)E_1  + (u-1)E_{1,3}  + (u-1)^2 E_1E_2\right]T_{12}$
\item $T_2T_1T_{12} =\left[ 1 + (u-1)E_2 + (u-1)E_{1,3} + (u-1)^2 E_1E_2\right]T_{12}$
\item $T_1T_2T_1 T_{12}=\left[ 1 + (u-1)(E_1 + E_2 + E_{1,3}) + (u-1)^2(u+2)E_1E_2\right]T_{12}$
\end{enumerate}
\end{lemma}
\begin{proof}
The proof of the statements results by expanding the  left side and then  using  the defining relations of ${\mathcal E}_n$. As example we shall check the first statement:
\begin{eqnarray*}
T_1 T_{12} & = & T_1 + T_1^2 + T_1T_2 + T_1^2T_2 + T_1T_2T_1 + T_1^2 T_2T_1\\
& = & 
T_1 + 1 + (u-1)E_1 +(u-1)E_1T_1 +T_1T_2+ T_2 \\
&\,  & 
+  (u-1)E_1T_2 +(u-1)E_1T_1T_2 + T_1T_2T_1 + T_2T_1\\
&\,  & 
+ (u-1)E_1T_2T_1 + (u-1)E_1T_1T_2T_1 =
T_{12} + (u-1)E_1T_{12}.
\end{eqnarray*}

\end{proof}

\begin{lemma}\label{rhoEI}
\begin{enumerate}
\item $\rho_3(T_{12}) = (u+1)A^2+3A +(u-1)AB + 1$
\item $\rho_3(E_{\{1,2,3\}}T_{12}) = (u+1)A^2+(u+2)AB + B^2$
\item  $\rho_3(E_IT_{12}) =  (u+1)A^2 +( u+1) AB + A + B$, for all $I\in \mathsf{P}(3)$ of cardinal $2$.
\end{enumerate}
\end{lemma}
\begin{proof}
The proof is only a routine of computations. We shall prove, as an   example, the third claim. Suppose $I=\{\{1,2\},\{3\}\}$, hence  $E_I = E_1$. Then, by linearity and  using the  example above we have
$$
\rho_3(E_IT_{12})  = B + A + AB + A^2 + A^2 + u AB + (u-1)A^2
$$
Hence we have proved the claim.
\end{proof}
\begin{theorem}
The Markov trace ${\rho}_n:{\mathcal E}_n \longrightarrow K( A, B)$ passes to ${\rm PTL}_n$ if only if $A= -B$ or $A=- B/(1+u)$.
\end{theorem}
\begin{proof}
From Corollary \ref{PTLquo} we have that ${\rho}_n$ pass to ${\rm PTL}_n$ if only if 
${\rho}_n(xE_1E_2T_{12}y) = 0$, for all $x,y\in {\mathcal E}_n$. Now,  by linearity  and  trace properties of ${\rho}_n$ follows that it is enough to study the conditions to have  ${\rho}_n(xE_1E_2T_{12}) = 0$, for all $x$ in  a linear basis of ${\mathcal E}_n$. We consider now the basis ${\Bbb S}_n$ of  ${\mathcal E}_n$, see Theorem \ref{basEn}. Using the rules that define ${\rho}_n$ we deduce that the 
computation of ${\rho}_n(xE_1E_2T_{12})$, for $x\in {\Bbb S}_n$, results   in 
a $K(A, B)$--linear combination of $\rho_3 (zE_1E_2T_{12})$ with $z\in {\Bbb S}_3$.  Now, $z$ is of the form $E_IT_w$, with $w\in S_3$ and $I\in \mathsf{P}({\bf 3})$;   since  $T_w$ commutes with $E_1E_2$  having in mind the Lemma \ref{TwT12} and the fact that $E_1E_2$ is the maxim element of  $\mathsf{P}({\bf 3}) $,  we obtain  that $z E_1E_2 T_{12}$ is a $K$--scalar multiple of $E_1E_2T_{12}$. 
Therefore, ${\rho}_n(xE_1E_2T_{12}y) = 0$, for all $x,y\in {\mathcal E}_n$ is equivalent to have $\rho_3( E_1E_2T_{12}) =0$. Now, from (2)Lemma \ref{rhoEI}, we  have  $\rho(E_1E_2T_{12})= 0$ is equivalent to $(u+1)A^2 + (u+2)AB + B^2=0$, then $A= -B$ or $A=- B/(1+u)$.

\end{proof}

\end{document}